\documentclass[12pt]{amsart}
\usepackage{mathtools}

\usepackage{amssymb} 

\usepackage{graphicx}

\usepackage{graphics}
\usepackage{amsmath}
\usepackage{amsthm}

\usepackage{enumerate}

\newtheorem{thm}{Theorem}

\newtheorem*{theorem}{Theorem}

\theoremstyle{definition}

\def\<{\langle}

\def\>{\rangle}

\def\eps{\varepsilon}

\def\Det{\mathrm{Det}\,}

\def\M3{M_3(\bbbc)}

\def\M3r{M_3(\bbbr)}

\def\diag{\mathrm{Diag}}

\def\bbbr{{\mathbb R}}

\def\bbbc{{\mathbb C}}

\newcommand{\R}{\mathbb{R}}

\newcommand*{\be}{\begin{equation}}

\newcommand*{\ee}{\end{equation}}

\newcommand*{\bit}{\begin{itemize}}
\newcommand*{\eit}{\end{itemize}}
\newcommand*{\ben}{\begin{enumerate}}
\newcommand*{\een}{\end{enumerate}}

\newcommand{\tr}{\mathrm{Tr}}
\newcommand{\dt}{\mathrm{Det}}

\newcommand{\norm}[1]{\left\|#1\right\|}

\newcommand*{\inner}[2]{\left<#1,\,#2\right>}

\newcommand*{\ler}[1]{\left(#1\right)}

\newcommand{\ba}{\begin{array}}
\newcommand{\ea}{\end{array}}

\newcommand{\fH}{\mathbb{H}}
\newcommand{\fP}{\mathbb{P}}

\newcommand{\E}{\mathbb{E}_2}
\newcommand{\adj}{\operatorname{adj}}

\begin{document}
\title{Continuous Jordan triple endomorphisms of $\fP_2$}

\author{LAJOS MOLN\'AR}
\address{MTA-DE "Lend\" ulet" Functional Analysis Research Group, Institute of Mathematics\\
         University of Debrecen\\
         H-4010 Debrecen, P.O. Box 12, Hungary}
\email{molnarl@science.unideb.hu}
\urladdr{http://www.math.unideb.hu/\~{}molnarl/}

\author{D\'ANIEL VIROSZTEK}
\address{Institute of Mathematics\\ 
         Budapest University of Technology and Economics\\
         H-1521 Budapest, Hungary} 
\email{virosz@math.bme.hu} 
\urladdr{http://www.math.bme.hu/~virosz}
\thanks{The first author was supported by the "Lend\" ulet" Program (LP2012-46/2012) of the Hungarian Academy of Sciences. The second author was partially supported by the Hungarian Scientific Research Fund (OTKA) Reg. No.  K104206}
\maketitle

\begin{abstract}
We describe the structure of all continuous Jordan triple endomorphisms of the set $\fP_2$ of all positive definite $2\times 2$ matrices thus completing a recent result of ours. We also mention an application concerning sorts of surjective generalized isometries on $\fP_2$ and, as second  application, we complete another former result of ours on the structure of sequential endomorphisms of finite dimensional effect algebras.
\end{abstract}

\section*{}

Recently, we have been very interested in the structure of so-called Jordan triple endomorphisms of the set of all positive definite matrices or, more generally, those of the positive definite cones in operator algebras. These are maps which are morphisms with respect to the operation of the Jordan triple product $(A,B)\mapsto ABA$ which is a well-known operation in ring theory. Our main reason for investigating those maps comes from the fact that they naturally appear in the study of surjective isometries and surjective maps preserving generalized distance measures between positive definite cones. For details see \cite{lm, ML13j, ML15b}.

In the paper \cite{lm} we have proved the following statement which appeared as Theorem 1 there. In what follows we denote by $\mathbb M_n$ the algebra of all $n\times n$ complex matrices and $\mathbb P_n$ stands for the cone of all positive definite matrices in $\mathbb M_n$. When we use the word "continuity" we mean the topology of the operator norm, in other word, spectral norm (or any other norm on the finite dimensional linear space $\mathbb M_n$). The usual trace functional and the determinant are denoted by $\tr$ and $\Det$, respectively, and ${}^{tr}$ stands for the transpose operation.

\begin{theorem}
Assume $n\geq 3$. Let $\phi:{\mathbb P}_n \to {\mathbb P}_n$ be a continuous map which is a Jordan triple endomorphism, i.e., $\phi$ is a continuous map which satisfies
\[
\phi(ABA)=\phi(A)\phi(B)\phi(A), \quad A,B\in {\mathbb P}_n.
\]
Then there exist a unitary matrix $U\in {\mathbb M}_n$, a real number $c$, a set $\{P_1,\ldots,P_n\}$ of mutually orthogonal rank-one projections in ${\mathbb M}_n$, and a set $\{c_1,\ldots,c_n\}$ of real numbers such that $\phi$ is of one of the following forms:
\begin{itemize}
\item[(a1)]
$\phi(A)=(\Det A)^c UAU^*$, \, $A\in {\mathbb P}_n$;
\item[(a2)]
$\phi(A)=(\Det A)^c UA^{-1}U^*$,\, $A\in {\mathbb P}_n$;
\item[(a3)]
$\phi(A)=(\Det A)^c UA^{tr}U^*$,\, $A\in {\mathbb P}_n$;
\item[(a4)]
$\phi(A)=(\Det A)^c U{A^{tr}}^{-1}U^*$,\, $A\in {\mathbb P}_n$;
\item[(a5)]
$\phi(A)=\sum_{j=1}^n (\Det A)^{c_j} P_j$,\, $A\in{\mathbb P}_n$.
\end{itemize}
\end{theorem} 

Observe that the converse statement in Theorem is also true meaning that any transformation of any of the forms (a1)-(a5) is necessarily a continuous Jordan triple endomorphism of $\mathbb P_n$.

One may immediately ask why we assume the condition $n\geq 3$, what happens in the case where $n=2$. The fact is that in the proof of Theorem we used such tools which are applicable  only if $n\geq 3$. Of course, we were very interested in the remaining case $n=2$ but unfortunately could not come up with a solution. Therefore, we proposed it as an open problem in our papers \cite{lm} (see Remark 11) and \cite{ML13j} (see Remark 23). 

One may think that when $n=2$, one can simply compute and obtain the solution straightaway. But this is far from being true as it will turn out below.
Indeed, the aim of this paper is to solve that problem and also present a few applications.

Our main result reads as follows.

\begin{thm}\label{T:main}
Let $\phi: \fP_2 \rightarrow \fP_2$ be a continuous Jordan-triple endomorphism. Then we have the following possibilities:
\begin{itemize}
\item[(b1)]
there is a unitary matrix $U\in \mathbb M_2$ and a real number $c$ such that $$\phi(A)=(\Det A)^c UAU^*, \quad A\in \fP_2;$$
\item[(b2)]
there is a unitary matrix $V\in \mathbb M_2$ and a real number $d$ such that $$\phi(A)=(\Det A)^d VA^{-1}V^*, \quad A\in \fP_2;$$
\item[(b3)]
there is a unitary matrix $W\in \mathbb M_2$ and real numbers $c_1,c_2$ such that 
$$
\phi(A)=W\diag [(\Det A)^{c_1}, (\Det A)^{c_2}]W^*, \quad A\in \fP_2.
$$
\end{itemize}
\end{thm}

Before presenting the proof we introduce a few notation and make some useful observations.

In what follows we denote by $\mathbb H_n$ the space of all self-adjoint elements of $\mathbb M_n$. 

We equip $\fH_2$ with the inner product $\inner{X}{Y}=(1/2)\tr XY$. The induced norm is denoted by $\norm{\cdot}.$ The set
\be \label{basis}
\left\{ \sigma_0=I=\left[ \ba{cc} 1 & 0 \\ 0 & 1 \ea\right], \, \sigma_x=\left[ \ba{cc} 0 & 1 \\ 1 & 0 \ea\right], \, \sigma_y=\left[ \ba{cc} 0 & i \\ -i & 0 \ea\right], \, \sigma_z=\left[ \ba{cc} 1 & 0 \\ 0 & -1 \ea\right] \right\}
\ee
is a convenient orthonormal basis in $\fH_2.$ Let $\fH_{2,0}$ denote the traceless subspace of $\fH_2$ (the subspace of all elements in $\fH_2$ with zero trace). 

In the proof of our theorem we shall use the following two observations. We first claim that for $X \in \fH_{2,0},$ the equality $X^2=I$ holds iff $\norm{X}=1.$
Indeed, let us denote the eigenvalues of $X$ by $\lambda$ and $-\lambda$, $\lambda \geq 0$. We have
$$
X^2=I \Leftrightarrow \lambda^2=1 \Leftrightarrow \frac{1}{2} \ler{\lambda^2+(-\lambda)^2}=1 \Leftrightarrow \norm{X}=1
$$
verifying our first calim. 

Next, we  assert that
for any $0\neq X \in \fH_{2,0}$ we have 
$e^X=\ler{\cosh{\norm{X}}} I + \ler{\sinh \norm{X}} ({X}/{\norm{X}})$.
To see this, using $({X}/{\norm{X}})^2=I$, we compute
$$
e^X=e^{\norm{X}\frac{X}{\norm{X}}}=\sum_{k=0}^{\infty} \frac{1}{k!}\norm{X}^k \ler{\frac{X}{\norm{X}}}^k=\sum_{k=0}^{\infty} \frac{1}{(2k)!}\norm{X}^{2k} I + \sum_{k=0}^{\infty} \frac{1}{(2k+1)!}\norm{X}^{2k+1} \frac{X}{\norm{X}}.
$$
This proves our assertion.

Now we turn to the proof of the main result.

\begin{proof}[Proof of Theorem~\ref{T:main}]
Let $\phi: \fP_2 \rightarrow \fP_2$ be a continuous Jordan triple endomorphism. Then, by \cite[Lemma 6]{lm} there exists a commutativity preserving linear transformation $f: \fH_2 \rightarrow \fH_2$ such that
$$
\phi(A)=\exp( f (\log A)), \quad A\in \fP_2.
$$
In fact, similar conclusion holds all for continuous Jordan triple endomorphisms between the positive definite cones of general $C^*$-algebras as it has been shown in \cite[Lemma 16]{ML13j}. By a commutativity preserving linear map we simply mean a transformation which sends commuting elements to commuting elements.

We have two possibilities for $f(I)$: It is either a scalar multiple of the identity or it is not. We divide the argument accordingly.

Assume first that $f(I)$ is not a scalar multiple of the identity. Then up to unitary similarity we may and do assume that $f(I)$ is a diagonal matrix with two different eigenvalues. By the commutativity preserving property of $f$, for every $A\in \fH_2$ we have that $f(A)$ commutes with $f(I)$ and then it follows that $f(A)$ is diagonal, too. Therefore, we have 
linear functionals $\varphi,\psi:\fH_2 \to \R$ such that
\[
f(A)=\left[ \ba{cc} \varphi(A) & 0 \\ 0 & \psi(A) \ea\right], \quad A\in \fH_2 
\]
and hence
\[
\phi(A)=\left[ \ba{cc} e^{\varphi(\log A)} & 0 \\ 0 & e^{\psi(\log A)} \ea\right], \quad A\in \fP_2.
\]
Since $\phi$ is a Jordan triple endomorphism, we deduce easily that
\[ 
\varphi(\log ABA)=2\varphi(\log A)+\varphi(\log B), 
\quad A, B\in \fP_2
\]
and similar equality holds for $\psi$ as well. Since $\varphi$ is a linear functional on $\fH_2$, by Riesz representation theorem we have an element $T\in \fH_2$ such that $\varphi(\cdot)=\inner{\cdot}{T}$. It follows that we have
\[ 
\tr((\log ABA)T)=2\tr((\log A)T)+\tr((\log B)T), 
\quad A, B\in \fP_2.
\]
Following the argument given on p. 2844 in \cite{ML06a} from the displayed equality (2) on, one can verify that $T$ is necessarily a scalar multiple of the identity and that means that $\varphi(A)=c\tr A$, $A\in \fH_2$ holds for some real number $c$. The same observation applies for $\psi$, too, and then we conclude that there are real numbers $c_1,c_2$ such that we have
\[
\phi(A)=\left[ \ba{cc} (\Det A)^{c_1} & 0 \\ 0 & (\Det A)^{c_2} \ea\right], \quad A\in \fP_2,
\]
which gives us (b3).

In the remaining part of the proof we assume that $f(I)$ 
is a scalar multiple of the identity.

Let us define the linear functional $f_0: \fH_2 \rightarrow \R$ by $f_0(\cdot)=\inner{f(\cdot)}{\sigma_0},$ that is, by $f_0(A)=(1/2)\tr f(A)$, $A\in \fH_2$.

The first crucial step in the proof follows.

\smallskip \smallskip

\textbf{Claim 1.}
\textit{The linear functional $f_0$ vanishes on $\fH_{2,0}.$}

The subspace $\fH_{2,0}$ is generated by $\sigma_x, \sigma_y, \sigma_z$.
We show that $f_0(\sigma_x)=f_0(\sigma_y)=0$, the remaining equality $f_0(\sigma_z)=0$ can be verified similarly. In what follows we consider arbitrary positive real parameters $s,t$. Direct calculations show that for all such $s,t$ we have
\begin{equation*}
\begin{gathered}
e^{\frac{s}{2} \sigma_x} e^{t \sigma_y} e^{\frac{s}{2} \sigma_x}
=\ler{\cosh\ler{\frac{s}{2}}I+\sinh\ler{\frac{s}{2}} \sigma_x} \ler{\cosh(t)I+\sinh(t)\sigma_y} \ler{\cosh\ler{\frac{s}{2}}I+\sinh\ler{\frac{s}{2}} \sigma_x}\\
=\cosh(t)\cosh^2 \ler{\frac{s}{2}} I+ 2 \cosh(t) \cosh\ler{\frac{s}{2}} \sinh\ler{\frac{s}{2}} \sigma_x+\cosh^2\ler{\frac{s}{2}} \sinh(t) \sigma_y\\
+\sinh\ler{\frac{s}{2}}\sinh(t)\cosh\ler{\frac{s}{2}} \sigma_x \sigma_y+\cosh\ler{\frac{s}{2}}\sinh(t)\sinh\ler{\frac{s}{2}} \sigma_y \sigma_x\\+ 
\cosh(t) \sinh^2\ler{\frac{s}{2}} \sigma_x^2
+\sinh^2 \ler{\frac{s}{2}} \sinh(t) \sigma_x \sigma_y \sigma_x\\
=\cosh(s)\cosh(t) I +\cosh(t)\sinh(s)\sigma_x+\sinh(t)\sigma_y.
\end{gathered}
\end{equation*}
Here we have used the equalities $\sigma_x \sigma_y+ \sigma_y \sigma_x=0, \, \sigma_x^2=I, \, \sigma_x \sigma_y \sigma_x= -\sigma_y$ and some identities of the hyperbolic functions. \par

Since, by the multiplicativity of the determinant, we have  $\dt\ler{e^{\frac{s}{2} \sigma_x} e^{t \sigma_y} e^{\frac{s}{2} \sigma_x}}=1,$ hence
$$e^{\frac{s}{2} \sigma_x} e^{t \sigma_y} e^{\frac{s}{2} \sigma_x}=e^{rW}$$
holds for some $W \in \fH_{2,0}$ with $\norm{W}=1$ and $r \geq 0$ (observe that $W$ depends on $s,t$). Since $e^{rW}=\cosh(r)I+\sinh(r)W$ we obtain the equality
\[
\cosh(r)I+\sinh(r)W=
\cosh(s)\cosh(t) I +\cosh(t)\sinh(s)\sigma_x+\sinh(t)\sigma_y.
\]
Taking trace we first deduce that 
\be \label{E:M2}
r=\cosh^{-1}\ler{\cosh(s)\cosh(t)}
\ee
and next that 
\[
\sinh(r)W=
\cosh(t)\sinh(s)\sigma_x+\sinh(t)\sigma_y.
\]
Clearly, due to $s,t>0$, the possibility $r=0$ is ruled out and hence we infer that
\begin{equation} \label{E:M1}
W=\frac{1}{\sinh(r)}\ler{\cosh(t)\sinh(s)\sigma_x+\sinh(t)\sigma_y}=\frac{\cosh(t)\sinh(s)\sigma_x+\sinh(t)\sigma_y}{\sqrt{\cosh^2(s)\cosh^2(t)-1}}.
\end{equation}
Now, on the one hand, we compute
\be \label{egyr}
\dt\ler{\phi\ler{e^{\frac{s}{2} \sigma_x} e^{t \sigma_y} e^{\frac{s}{2} \sigma_x}}}=\dt\ler{e^{f \ler{\log \ler{e^{\frac{s}{2} \sigma_x} e^{t \sigma_y} e^{\frac{s}{2} \sigma_x}}}}}
=e^{\tr f \ler{\log \ler{e^{\frac{s}{2} \sigma_x} e^{t \sigma_y} e^{\frac{s}{2} \sigma_x}}}}=e^{2 f_0(r W)}.
\ee
On the other hand, since $\phi$ is a Jordan triple endomorphism, the quantity (\ref{egyr}) is equal to
\begin{equation} \label{masr}
\begin{gathered}
\dt\ler{\phi\ler{e^{\frac{s}{2} \sigma_x}}\phi\ler{e^{t \sigma_y}} \phi\ler{e^{\frac{s}{2} \sigma_x}}}=
\dt\ler{e^{\frac{s}{2} f (\sigma_x)} e^{t f (\sigma_y)} e^{\frac{s}{2} f (\sigma_x)}}\\
=e^{\frac{s}{2} \tr f (\sigma_x)}e^{t \tr f (\sigma_y)}e^{\frac{s}{2} \tr f (\sigma_x)}=e^{2\ler{s f_0(\sigma_x)+t f_0(\sigma_y)}}.
\end{gathered}
\end{equation}
Let us introduce the auxiliary function 
$$
N(s,t)=\frac{\cosh^{-1}\ler{\cosh(s)\cosh(t)}}{\sqrt{\cosh^2(s)\cosh^2(t)-1}}, \quad 0<s,t\in\R.
$$
By \eqref{E:M2}, \eqref{E:M1}, \eqref{egyr}, \eqref{masr} we have
$$
s f_0(\sigma_x)+t f_0(\sigma_y)=f_0(rW)=N(s,t)\cosh(t)\sinh(s)f_0(\sigma_x)+N(s,t)\sinh(t)f_0(\sigma_y)
$$
for all $0<s,t \in \R.$ It is not difficult to check that the two-variable functions
$g(s,t)=N(s,t)\cosh(t)\sinh(s)-s$ and $h(s,t)=N(s,t)\sinh(t)-t$ are linearly independent. Indeed, one can see that the determinant of the matrix
$$
\left[ \ba{cc} g(1,1) & h(1,1) \\ g(2,2) & h(2,2) \ea\right]
$$
is nonzero (its value is close to -0.5) which implies the desired linear independence.
It then follows that $f_0(\sigma_x)=f_0(\sigma_y)=0$ and we obtain Claim 1.

\smallskip \smallskip

As a consequence we infer that the subspace $f(\fH_{2,0})$ is orthogonal to $\sigma_0=I$ meaning that it consists of traceless matrices,
$f(\fH_{2,0}) \subset \fH_{2,0}$. Since $f(I)$ is  a scalar multiple of the identity, we also have $f(\fH_{2,0}^\bot) \subset \fH_{2,0}^\bot$. We will use these facts in the second crucial step of the proof which follows.

\smallskip \smallskip

\textbf{Claim 2.}
\textit{The restriction of $f$ to $\fH_{2,0}$ is a non-negative scalar multiple of an isometry.}

To see this, it is sufficient to show that $f(\sigma_x), f(\sigma_y), f(\sigma_z)$ are mutually orthogonal and of the same norm. Clearly, we are done if we verify this for any two elements of the collection $f(\sigma_x), f(\sigma_y), f(\sigma_z)$. We shall consider, for example, $f(\sigma_x)$ and $f(\sigma_y)$. Recalling that $f(W)$ is traceless, in the case where $f(W)\neq 0$, we compute
\[
\begin{gathered}
l(s,t):=\frac{1}{2} \tr \phi\ler{e^{\frac{s}{2} \sigma_x} e^{t \sigma_y} e^{\frac{s}{2} \sigma_x}}= \frac{1}{2} \tr \ler{e^{f \ler{\log \ler{e^{\frac{s}{2} \sigma_x} e^{t \sigma_y} e^{\frac{s}{2} \sigma_x}}}}}=\frac{1}{2} \tr e^{r f(W)}\\
= \frac{1}{2} \tr \ler{\cosh\ler{r \norm{f(W)}}I+\sinh\ler{r \norm{f(W)}}\frac{f(W)}{\norm{f(W)}}}=\cosh\ler{r \norm{f(W)}}.
\end{gathered}
\]
If $f(W)=0$, then we again have $l(s,t)=\cosh\ler{r \norm{f(W)}}$ and, by \eqref{E:M1}, we can further compute
\begin{equation} \label{bla1}
\begin{gathered}
l(s,t)=\cosh\ler{\norm{f(W)} \cosh^{-1}\ler{\cosh(s)\cosh(t)}}
\\=\cosh \Biggl(\frac{\cosh^{-1}\ler{\cosh(s)\cosh(t)}}{\sqrt{\cosh^2(s)\cosh^2(t)-1}} 
\times \sqrt{\sinh^2(s)\cosh^2(t)\norm{f (\sigma_x)}^2}
\\ \overline{+\inner{f (\sigma_x)}{f (\sigma_y)} 2 \sinh(s)\sinh(t)\cosh(t)+\sinh^2(t) \norm{f (\sigma_y)}^2} \Biggr)
\\=\cosh\Biggl(\frac{\cosh^{-1}\ler{\cosh(s)\cosh(t)}}{\sqrt{\cosh^2(s)\cosh^2(t)-1}} \times \sqrt{\ler{\cosh^2(s)\cosh^2(t)-\cosh^2(t)}\norm{f (\sigma_x)}^2}
\\ 
\overline{+\inner{f (\sigma_x)}{f (\sigma_y)} 2 \sinh(s)\sinh(t)\cosh(t)+\ler{\cosh^2(t)-1} \norm{f ( \sigma_y)}^2} \Biggr).
\end{gathered}
\end{equation}
Since $\phi$ is a Jordan triple endomorphism,
(\ref{bla1}) is equal to
\[
\begin{gathered}
m(s,t):=\frac{1}{2} \tr \ler{\phi\ler{e^{\frac{s}{2} \sigma_x}}\phi\ler{e^{t \sigma_y}} \phi\ler{e^{\frac{s}{2} \sigma_x}}}=
\frac{1}{2} \tr \ler{e^{\frac{s}{2} f (\sigma_x)} e^{t f( \sigma_y)} e^{\frac{s}{2} f (\sigma_x)}}
=\frac{1}{2} \tr \ler{e^{s f (\sigma_x)} e^{t f (\sigma_y)}}.
\end{gathered}
\]
Assume $f(\sigma_x), f(\sigma_y)\neq 0$ and denote $X=f(\sigma_x)/\norm{f(\sigma_x)}$ and $Y=f(\sigma_y)/\norm{f(\sigma_y)}$. Then, since $f(\sigma_x), f(\sigma_y)$ are traceless, we can continue
\begin{equation}\label{bla2}
\begin{gathered}
m(s,t)
=\frac{1}{2} \tr \ler{\cosh\ler{s \norm{f (\sigma_x)}} I+\sinh\ler{s \norm{f (\sigma_x)}}\frac{f(\sigma_x)}{\norm{f(\sigma_x)}}}\\ \times \ler{\cosh\ler{t \norm{f( \sigma_y)}} I+\sinh\ler{t \norm{f (\sigma_y)}}\frac{f(\sigma_y)}{\norm{f(\sigma_y)}}}\\
=\cosh\ler{s \norm{f (\sigma_x)}}\cosh\ler{t \norm{f( \sigma_y)}}+\inner{X}{Y}\sinh\ler{s \norm{f (\sigma_x)}}\sinh\ler{t \norm{f (\sigma_y)}}.
\end{gathered}
\end{equation}
We show that $\norm{f(\sigma_x)}=\norm{f(\sigma_y)}.$
To this, set $\alpha:=\norm{f (\sigma_x)}, \, \beta:=\norm{f (\sigma_y)}, \, \gamma:=\inner{X}{Y}.$
It is easy to check that
$$
\lim_{t \to \infty} \frac{1}{t}\cosh^{-1}\ler{\cosh^{2}(t)}=2
$$
and
$$
\lim_{t \to \infty}
\sqrt{\frac{\ler{\cosh^4(t)-\cosh^2(t)}\alpha^2+ 2 \alpha \beta \gamma \sinh^2(t)\cosh(t)+\ler{\cosh^2(t)-1} \beta^2}{\cosh^4(t)-1}}=\alpha.
$$
From these we get that for every $0 < \eps (<2)$ there exists some $0<T_\eps$ such that $l(t,t)\geq \cosh\ler{(2-\eps)\alpha t}$ holds for $t>T_\eps.$
On the other hand, it is easy to see that
$$
m(t,t)=\frac{1}{4}e^{(\alpha+\beta)t}\ler{1+\gamma+o(1)}.
$$
Therefore, the inequality
$$
m(t,t)=l(t,t)\geq \cosh\ler{(2-\eps)\alpha t}
$$
is equivalent to
\be \label{csil}
\frac{1}{4}\ler{1+\gamma+o(1)} \geq \frac{1}{2}\ler{e^{((1-\eps)\alpha-\beta) t}+e^{-((3-\eps)\alpha+\beta) t}}
\ee
Taking the limit ${t\to \infty}$ in (\ref{csil}) we infer that $\beta \geq (1-\eps)\alpha.$ This is true for any $0<\eps<2,$ hence letting $\epsilon \to 0$ we obtain $\beta\geq\alpha,$ that is, $\norm{f (\sigma_y)} \geq \norm{f (\sigma_x)}.$ By changing the roles of $\sigma_x$ and $\sigma_y$ we get the desired equality $\norm{f (\sigma_y)} = \norm{f (\sigma_x)}$.
Having this in mind, it is clear that the function $m(\cdot, \cdot)$ is symmetric in the sense that we have $m(s,t)=m(t,s)$ for all $0<s,t \in \R$, see (\ref{bla2}).
It follows that $l(\cdot, \cdot)$ is also symmetric which can happen only when $\inner{f (\sigma_x)}{f (\sigma_y)}=0$, see (\ref{bla1}). Therefore, we have $\| f(\sigma_x)\|=\|f(\sigma_y)\|$,  $\inner{f (\sigma_x)}{f (\sigma_y)}=0$ and we are done in the case where $f(\sigma_x), f(\sigma_y)\neq 0$.

Assume now that $f(\sigma_x)=0, f(\sigma_y)\neq 0$.
By \eqref{bla1}, \eqref{bla2} we have
$$
\cosh\left(\frac{\cosh^{-1}\ler{\cosh(s)\cosh(t)}}{\sqrt{\cosh^2(s)\cosh^2(t)-1}} \sqrt{\ler{\cosh^2(t)-1} \norm{f ( \sigma_y)}^2} \right)=
\cosh\ler{t \norm{f (\sigma_y)}}.
$$
It follows that
$$
\frac{\cosh^{-1}\ler{\cosh^2(t)}}{t}
\sqrt{\frac{\ler{\cosh^2(t)-1}\norm{f ( \sigma_y)}^2}{\cosh^4(t)-1} } =
\norm{f (\sigma_y)}.
$$
Letting $t$ tend to infinity, we obtain $f(\sigma_y)=0$, a contradiction. 

Assume $f(\sigma_x)\neq 0, f(\sigma_y)= 0$.
Again, by \eqref{bla1}, \eqref{bla2} we have
$$
\cosh\left(\frac{\cosh^{-1}\ler{\cosh(s)\cosh(t)}}{\sqrt{\cosh^2(s)\cosh^2(t)-1}} \sqrt{\ler{\cosh^2(s)\cosh^2(t)-\cosh^2(t)} \norm{f ( \sigma_x)}^2} \right)=
\cosh\ler{s \norm{f (\sigma_x)}}.
$$
It follows that
$$
\frac{\cosh^{-1}\ler{\cosh^2(t)}}{t}
\sqrt{\frac{\ler{\cosh^4(t)-\cosh^2(t)}\norm{f ( \sigma_x)}^2}{\cosh^4(t)-1} } =
\norm{f (\sigma_x)}.
$$
Letting $t$ tend to infinity, we deduce $2\norm{f(\sigma_x)}=\norm{f(\sigma_x)}$, i.e., $\norm{f(\sigma_x)}=0$, a contradiction again. So it remains only the possibility $f(\sigma_x)=f(\sigma_y)= 0$ and this proves Claim 2.

\smallskip \smallskip

To complete the proof of our theorem,
let us see what happens when the restriction of $f$ onto $\fH_{2,0}$ is zero. We have $f(I)=(2c)I$ with some real number $c$. Then $f(A)=c(\tr A) I$, $A\in \fH_2$ and we obtain $\phi(A)=(\Det A)^c I$, $A\in \fP_2$. This means that $\phi$ is of the form (b3).

Now assume that the restriction of $f$ onto $\fH_{2,0}$ is a positive scalar multiple of an isometry.
It follows that in the orthonormal basis (\ref{basis}), the transformation $f$ has the block-matrix form
$$
f=p \left[\ba{cc} v & 0 \\ 0 & M \ea \right],
$$
where $p$ is a positive real number, $v$ is a real number and $M$ is a $3 \times 3$ orthogonal matrix.

If $M \in \mathbf{SO}(3)$, then
$$f=p \left[\ba{cc} 1+2c & 0 \\ 0 & R \ea \right]$$ for some $c \in \R$ and $R \in \mathbf{SO}(3).$ 
Similarly,
if $-M \in \mathbf{SO}(3)$, then
$$f=p \left[\ba{cc} -1+2c & 0 \\ 0 & -R \ea \right]$$ for some $c \in \R$ and $R \in \mathbf{SO}(3).$ 

For any $R \in \mathbf{SO}(3)$ there exists a $U \in \mathbf{SU}(2)$ such that the matrix of the transformation $A \mapsto UAU^*$ is
$$\left[\ba{cc} 1 & 0 \\ 0 & R \ea \right],$$ see \cite[Proposition VII.5.7.]{sim}.
Therefore, in the case where $M \in \mathbf{SO}(3)$ we
have
$$
\begin{gathered}
\phi(A)= \exp (f (\log(A)))=
\exp (f(\log A-(\tr (\log A)/2)I))+f((\tr (\log A)/2)I))\\=
\exp(pU(\log A-(\tr (\log A)/2)I)U^*)\exp(p(1+2c)\tr (\log A)/2)\\=
\exp(pU(\log A)U^*)\exp((pc)\tr (\log A))
=(\dt A)^{p c} U A^p U^*.
\end{gathered}
$$
Since $\phi(ABA)=\phi(A)\phi(B)\phi(A)$, we infer $(ABA)^p=A^pB^pA^p$, $A,B\in \fP_2$ which holds only if $p=1$. Consequently, we have 
$\phi(A)=(\dt A)^{c} U A U^*$, $A\in \fP_2$. This means that $\phi$ is of the form (b1).
Similarly, in the case where
$-M \in \mathbf{SO}(3)$ one can conclude 
$\phi(A)=(\dt A)^{c} U A^{-1} U^*$, $A\in \fP_2$, i.e, 
$\phi$ is of the form (b2). The proof of the theorem is complete.
\end{proof}

One can notice that in Theorem describing the structure of continuous Jordan triple endomorphisms of $\fP_n$, in the case where $n\geq 3$ the transpose operation and its composition with the inverse operation also appear and one may ask why it is not so in the case where $n=2$. There is no contradiction here, it is easy to see that in fact those two possibilities do appear in Theorem~\ref{T:main} in a hidden way. Indeed, when $n=2$, the transpose operation can be written in the form (a2) above. Namely, for the unitary matrix
\[
U=
\left[ \ba{cc} 0 & 1 \\ -1 & 0 \ea\right]
\]
we have $A^{tr}=(\Det A) U A^{-1} U^*$ for all $A\in \fP_2$.

The following structural result 
concerning the continuous Jordan triple automorphisms of $\fP_2$ follows from the proof of Theorem~\ref{T:main}. 

\begin{thm}\label{T:M2}
 If $\phi: \fP_2 \rightarrow \fP_2$ is a continuous Jordan triple automorphism, then $\phi$ is of one of the following two forms:
 \begin{itemize}
  \item[(c1)]  there is a real number $c\neq -1/2$ and $U \in \mathbf{SU}(2)$ such that
  $$\phi(A)=(\dt A)^c U A U^*, \quad A\in \fP_2;$$ 
  \item[(c2)] there is a real number $d\neq 1/2$ and $V \in \mathbf{SU}(2)$ such that 
  $$\phi(A)=(\dt A)^d V A^{-1} V^*, \quad A\in \fP_2.$$
 \end{itemize}
\end{thm} 

The result above has the following immediate consequence.
In the case where $n\geq 3$, in \cite[Theorem 1]{ML15b} we obtained a general result describing the possible structure of surjective maps on $\fP_n$ which preserve a generalized distance measure of a certain quite general kind. It is easy to see that, following the proof of \cite[Theorem 1]{ML15b} and applying Theorem~\ref{T:M2}, the result in \cite{ML15b} remains valid also in the case where $n=2$.

\smallskip

In the rest of the paper
we present an application of Theorem~\ref{T:main} for the description of so-called sequential endomorphisms of effect algebras.

Effects play an important role in certain parts of quantum mechanics, for instance, in the quantum theory of measurement \cite{BusLahMit91}. Mathematically, effects are represented by positive semi-definite Hilbert space operators which are bounded (in the natural order $\leq $ among self-adjoint operators) by the identity. The set of all Hilbert space effects are called the Hilbert space effect algebra (although it is clearly not an algebra in the classical algebraic sense).
In \cite{GudNag01} Gudder and Nagy introduced the operation $\circ$ called sequential product on effects which has an important physical a meaning and which is closely related the Jordan triple product. Namely, they defined 
\[
A\circ B=A^{1/2} B A^{1/2}
\]
for arbitrary Hilbert space effects $A,B$. The corresponding endomorphism, i.e., maps $\phi$ on Hilbert space effects
which satisfy
\[
\phi(A\circ B)=\phi(A)\circ \phi(B)
\]
for all pairs $A,B$ of effects are called sequential endomorphisms.
In the literature one can find results related to sequential automorphisms or isomorphisms (bijective sequential endomorphisms). For example,
Gudder and Greechie proved in
\cite[Theorem 1]{GudGre02} that, supposing the dimension of the underlying Hilbert space is at least 3, the sequential automorphisms of the Hilbert space effect algebra are exactly the transformations $\phi$ which are of the form
$\phi: A\mapsto UAU^*$,
where $U$ is either a unitary or an antiunitary operator on the underlying Hilbert space. As a byproduct of one of our results concerning certain preserver transformations on Hilbert space effects, in \cite[Corollary 7]{ML03d}
we obtained that the latter result holds also in the 2-dimensional case.
Afterwards, in \cite{ML03g} we substantially generalized the previous results and described the structure of sequential isomorphisms between von Neumann algebra effects (i.e., between sets of effects on Hilbert spaces belonging to given von Neumann algebras).

In the paper \cite{ML12c} we studied sequential endomorphisms of effect algebras over finite dimensional Hilbert spaces of dimension at least 3. Anybody can easily be convinced that the problem of describing non-bijective morphisms is usually much harder than that of describing bijective ones. In \cite[Theorem 1]{ML12c} we managed to give the precise description of all continuous sequential endomorphisms assuming the dimension is at least 3.
However, the 2-dimensional case remained unresolved and in  \cite[Remark 6]{ML12c} we proposed it as an open problem. Now, using the main result of the present paper we can present a solution of the problem. 

For any positive integer $n$ denote by $\mathbb E_n$ the set of all positive semi-definite $n\times n$ matrices $A$ which satisfy $A\leq I$ (recall that in the natural order $\leq$ on self-adjoint matrices we have $A\leq B$ iff $B-A$ is positive semi-definite).

\begin{thm}\label{T:M1}
Assume $\phi: \E \to \E$ is a continuous sequential endomorphism. Then we have the following four possibilities:

\noindent
\begin{itemize}
\item[(d1)]
there exists a unitary $U\in \mathbb M_2$ and a non-negative real number $c$ such that
\begin{equation*}\label{F:4}
\phi(A)=(\det A)^c UAU^*, \quad A\in \E;
\end{equation*}
\item[(d2)]
there exists a unitary $V\in \mathbb M_2$ such that
\begin{equation*}\label{F:5}
\phi(A)=V(\adj A )V^*, \quad A\in \E;
\end{equation*}
\item[(d3)]
there exists a unitary $V\in \mathbb M_2$ and a real number $d>1$ such that
\begin{equation*}\label{F:6}
\phi(A)=
\left\{
  \begin{array}{ll}
    (\det A)^d VA^{-1}V^*, & \hbox{if $A\in \E$ is invertible;} \\
    0, & \hbox{otherwise;}
  \end{array}
\right.
\end{equation*}
\item[(d4)]
there exists a unitary $W\in \mathbb M_2$ and non-negative real numbers $c_1, c_2$ such that
\begin{equation*}\label{F:7}
\phi(A)= W\diag [(\Det A)^{c_1}, (\Det A)^{c_2}]W^*, \quad A\in \E.
\end{equation*}
\end{itemize}
Here, we mean $0^0=1$.
\end{thm}

\begin{proof}
First observe that every sequential endomorphism $\phi:\E \to \E$ is automatically a Jordan triple map. Indeed,
we clearly have $\phi(A^2)=\phi(A)^2$, $A\in \E$. 
It implies that $\phi(\sqrt{A})=\sqrt{\phi(A)}$, $A\in \E$ and hence it follows that $\phi$ is a Jordan triple map, i.e., $\phi$ satisfies
$\phi(ABA)=\phi(A)\phi(B)\phi(A)$, $A,B\in \E$. Moreover,
we infer that $\phi$ sends projections to projections implying that $\phi(I)$ is a projection. If $\phi(I)=0$, we easily have that $\phi$ is identically zero. If $\phi(I)=P$ is a rank-one projection, then by $\phi(A)=\phi(IAI)=P\phi(A)P$ it follows the map $A\mapsto \phi(A)+(I-P)$, $A\in \E$ is a sequential endomorphism of $\E$ which is unital, i.e., it sends $I$ to $I$. 

Therefore, in what follows we may and do assume that our original transformation $\phi$ is a continuous unital sequential endomorphism (and hence a Jordan triple map).

Consider the function $\lambda \mapsto \Det \phi(\lambda I)$, $\lambda\in [0,1]$. Clearly, this is a continuous multiplicative map of the interval $[0,1]$ into itself which sends 1 to 1. Lemma 3 in  \cite{ML12c} tells us that such a function is either everywhere equal to 1 or it is a power function corresponding to a positive exponent. This means that $\phi(\lambda I)$ is invertible for all $0<\lambda \leq 1$.
We claim that $\phi$ sends invertible elements of $\E$ to invertible elements. To see this, first observe that $\phi$ preserves the usual order $\leq$. Indeed, by \cite[Theorem 5.1]{GudGre02b} we know that for any $A,B\in \E$ we have $A\leq B$ if and only if there is a $C\in \E$ such that $A=B\circ C$. This clearly shows that for any $A,B\in \E$ with $A\leq B$ we have $\phi(A)\leq \phi(B)$.
Now, if $A\in \E$ is invertible, then there is a scalar $0<\lambda\leq 1$ such that $\lambda I\leq A$ holds which implies that $\phi(\lambda I)\leq \phi(A)$. Since $\phi(\lambda I)$ is invertible, it follows that $\phi(A)$ is also invertible.

The sequential endomorphism $\phi$ preserves commutativity. This follows from the fact that for any pair $A,B$ of effects we have $A\circ B=B\circ A$ if and only $A,B$ as matrices commute (see, e.g., Corollary 2.2 in \cite{GudNag01}). It follows that the effects $\phi(\lambda I)$, $\lambda\in [0,1]$ all commute and hence they are jointly diagonizable. This means that up to unitary similarity we can write
\[
\phi(\lambda I)=
\left[ \ba{cc} \varphi(\lambda ) & 0 \\ 0 & \psi(\lambda) \ea\right], \quad \lambda \in[0,1]
\]
where $\varphi,\psi:[0,1] \to [0,1]$ are continuous multiplicative functions which send 1 to 1. Therefore, by \cite[Lemma 3]{ML12c} again, we have real numbers $c,d\geq 0$ such that
\[
\phi(\lambda I)=
\left[ \ba{cc} \lambda^c & 0 \\ 0 & \lambda^d \ea\right], \quad \lambda \in[0,1].
\]

We now distinguish  two cases.
Assume first that there is $\phi(A)$ which is not diagonal. Since $\phi(A)$ necessarily commute with $\phi(\lambda I)$, $\lambda \in [0,1]$, one can easily deduce that we necessarily have $c=d$. It follows that  $\phi(\lambda I)=\lambda^c I$ and hence we have $\phi(\lambda A)=\lambda^c \phi(A)$ for all $\lambda \in[0,1], A\in \E$.

We next define $\Phi:\fP_2 \to \fP_2$ by
\begin{equation}\label{E:M3}
\Phi(A)=\| A\|^c\phi(A/\| A\|), \quad  A\in \E.
\end{equation}
In contrast to the proof of our main result, $\|.\|$ stands here for the operator norm (spectral norm) of matrices; we do hope it does not cause serious confusion. It follows that for any invertible effect $A\in \E$ we have
\[
\Phi(A)=\| A\|^c\phi(A/\| A\|)=\phi(\| A\| (A/\| A\|))=\phi(A).
\]
We assert that $\Phi$ is a Jordan triple endomorphism of $\fP_2$. Indeed, for any $A,B\in \fP_2$ we compute
\[
\begin{gathered}
\Phi(A)\Phi(B)\Phi(A)=\| A\|^{2c}\|B\|\phi\biggl(\frac{A}{\|A\|}\biggr)\phi\biggl(\frac{B}{\|B\|}\biggr)\phi\biggl(\frac{A}{\|A\|}\biggr)\\
=\| A\|^{2c}\|B\|\phi\biggl(\frac{ABA}{\|A\|\|B\|\|A\|}\biggr)=\| A\|^{2c}\|B\|\phi\biggl(\frac{\|ABA\|}{\|A\|\|B\|\|A\|}\frac{ABA}{\|ABA\|}\biggr)\\
=\| A\|^{2c}\|B\|^c \biggl(\frac{\|ABA\|}{\|A\|\|B\|\|A\|}\biggr)^c
\phi\biggl(\frac{ABA}{\|ABA\|}\biggr)
=\|ABA\|^c \phi\biggl(\frac{ABA}{\|ABA\|}\biggr)=\Phi(ABA),
\end{gathered}
\]
where we have used the facts that $\|ABA\|/(\|A\|\|B\|\|A\|)\leq 1$ and that $(ABA)/\|ABA\|$ is an effect.
Clearly, $\Phi$ is continuous and hence Theorem~\ref{T:main} applies and we obtain that $\Phi$ is of one of the forms (b1), (b2). In the case of (b1), we have
that $\phi(A)=(\Det A)^c UAU^*$ holds for all invertible $A\in \E$ with a given unitary matrix $U$ and real number $c$. Since $\phi$ sends effects to effects, it follows easily that $c$ is necessarily non-negative. By continuity we deduce
$$
\phi(A)=(\Det A)^c UAU^*, \quad A\in \E
$$
yielding the possibility (d1).
Consider now the case where $\phi(A)=(\Det A)^d VA^{-1}V^*$ holds for all invertible $A\in \E$ with a given unitary matrix $V$ and real number $d$. Again, since $\phi$ sends effects to effects, one can easily verify that $d\geq 1$. 
If $d=1$, then we have
$$
\phi(A)=V (\adj A) V^*
$$
for all invertible $A\in \E$ and by continuity it follows that the same formula remains valid for any $A\in \E$, too. This gives us (d2). Assume $d>1$. Letting $A$ be an invertible effect tending to some non-invertible one, it follows that $\phi(A)=(\Det A)^d VA^{-1}V^*$ tends to 0. Hence, we obtain
that
$$
\phi(A)=
\left\{
  \begin{array}{ll}
    (\det A)^d VA^{-1}V^*, & \hbox{if $A\in \E$ is invertible;} \\
    0, & \hbox{otherwise}
  \end{array}
\right.
$$
and this is the possibility (d3).

It remains to discuss the case where
all $\phi(A)$ are diagonal, that is when we have 
\[
\phi(A)=
\left[ \ba{cc} \varphi(A) & 0 \\ 0 & \psi(A) \ea\right], \quad A\in \E
\]
for continuous (unital) Jordan triple maps $\varphi, \psi:\E\to [0,1]$. As in \eqref{E:M3}, we can extend $\varphi,\psi$ from the set of all invertible elements of $\E$ to continuous Jordan triple functionals 
$\tilde \varphi,\tilde \psi :\fP_2 \to ]0,\infty [$. Applying Theorem~\ref{T:main}, it follows that $\tilde \varphi,\tilde \psi$ are non-negative powers of the determinant function.
Hence we obtain that
\[
\phi(A)=
\left[ \ba{cc} (\Det A)^c & 0 \\ 0 & (\Det A)^d \ea\right], \quad A\in \E
\]
holds for some non-negative real numbers $c,d$. This gives (d4) and the proof of the theorem is complete.
\end{proof}

We conclude the paper with the following remark.
In \cite[Theorem 1]{ML12c} we considered effects as linear operators and the satement was formulated accordingly. One can notice that the operators $U,V,W$ in \cite{ML12c} were either unitaries or antiunitaries. However, in our present result Theorem~\ref{T:M1} only unitary matrices appear. The reason for this is the following. For an antiunitary $U$, the transformation $A\mapsto UA^* U^*$ is a linear antiautomorphism which hence can be written in the form $A\mapsto U' A^{tr} {U'}^*$ with some unitary $U$. But, as we have already seen, $A^{tr}=(\Det A) U'' A^{-1} {U''}^*$ holds for all $A\in \fP_2$ with some $2\times 2$ unitary matrix $U''$. That means that we have $A^{tr}= U'' (\adj A) {U''}^*$ for all $A\in \E$. One can now readily verify that if any of $U,V,W$ in Theorem~\ref{T:M1} would be an "antiunitary matrix" the corresponding map could still be written in one of the forms (d1)-(d3) with an appropriate unitary matrix.

\end{document}